\documentclass[preprint,3p,times,sort&compress]{elsarticle}

\usepackage{amssymb}
\usepackage{amsmath}         
\usepackage{mathrsfs}        
\usepackage{enumerate}       
\usepackage[colorlinks,
            linkcolor=red,
            anchorcolor=blue,
            citecolor=green
            ]{hyperref}
\usepackage{amsthm}          

\newtheorem{theorem}{Theorem}
\newtheorem{remark}{Remark}

\newtheorem{proposition}{Proposition}
\usepackage{caption2}                   
\journal{a journal}

\begin{document}

\begin{frontmatter}

\phantomsection
\addcontentsline{toc}{chapter}{Time--domain response of nabla discrete \\ fractional order systems}
\title{Time--domain response of nabla discrete \\ fractional order systems}


\author[ustc]{Yiheng Wei}
\author[ude]{Qing Gao}
\author[ustc]{Songsong Cheng}
\author[ustc]{Yong Wang\corref{cor1}}
\address[ustc]{Department of Automation, University of Science and Technology of China, Hefei, 230026, China}
\address[ude]{Faculty of Engineering, University of Duisburg-Essen, D-45117 Essen, Duisburg, Germany}

\cortext[cor1]{Corresponding author. E-mail: yongwang@ustc.edu.cn.}

\begin{abstract}
This paper investigates the time--domain response of nabla discrete fractional order systems by exploring several useful properties of the nabla discrete Laplace transform and the discrete Mittag--Leffler function. In particular, we establish two fundamental properties of a nabla discrete fractional order system with nonzero initial instant: i) the existence and uniqueness of the system time--domain response; and ii) the dynamic behavior of the zero input response. Finally, one numerical example is provided to show the validity of the theoretical results.
\end{abstract}


\end{frontmatter}


\section{Introduction}\label{Section 1}
Fractional calculus is a natural generalization of classical calculus and its inception can be traced back to a correspondence between Leibniz and L'H\^{o}pital in 1695 regarding the possible value of a half order derivative \cite{Oldham:1974Book}, i.e., $\frac{{{{\rm{d}}^{0.5}}f(x)}}{{{\rm{d}}{x^{0.5}}}}$. After over 300 years of development, fractional calculus has been widely used in many branches of science and engineering, and is particularly suitable for modeling physical plants that behave anomalously such as viscoelastic material \cite{Bagley:1983JR} and diffusion processes \cite{Metzler:2000PR} (see \cite{Chen:2017AMC,Sun:2018CNSNS,West:2016Book} and their references for  many examples).

Despite the important success achieved by continuous fractional calculus \cite{Liu:2015ISA,Wei:2017FCAAa,Sheng:2018ISA}, research on discrete fractional calculus is still not mature. In 1974, Diaz and Osler initiated a study of discrete fractional calculus by introducing an infinite series as the $\alpha$-th difference operation which is a generalization of the $m$-th difference operation \cite{Diaz:1974MC}. Unfortunately, the method in \cite{Diaz:1974MC} requires calculation of an infinite series which is often time-consuming, or even worse, infeasible. A more useful version of discrete fractional calculus was then proposed by Granger and Joyeux in \cite{Gray:1988MC} where the infinite series was replaced by a finite one. So far, a large volume of pioneering works on discrete fractional calculus have been reported, see, e.g., \cite{Anastassiou:2010MCM,Yucra:2013TAC,Tseng:2015TCSI,Cui:2018ISA,Liu:2019SP}. For a more comprehensive introduction on the most recent advances in relevant field, the readers can refer to the monographs \cite{Goodrich:2015Book,Ostalczyk:2015Book,Cheng:2011Book} and the papers cited therein for details. Remarkably, a $Z$-transform based approach was proposed to analyze the solution of a class of fractional difference equations by Cheng in \cite{Cheng:2011Book}. However, several fundamental dynamic properties of the solution have not been discussed yet. Moreover, the approach in \cite{Cheng:2011Book} can only describe the solution at each sampling point. To analyze the solution without the process of discretization, the authors in \cite{Atici:2009EJQTDE} proposed an $N$-transform approach which can be regarded as a discrete analog of the conventional Laplace transform. Based on this $N$-transform, the authors in \cite{Abdeljawad:2012ADE} presented an explicit form of the solution to a scalar fractional difference equation with the derivative order $\alpha\in (0,1)$, by using the semigroup property of the discrete Mittag--Leffler function; a similar but more rigorous treatment can be found in \cite{Mohan:2014CMS}. The convergence of this solution in \cite{Abdeljawad:2012ADE,Mohan:2014CMS} was investigated in \cite{Tan:2015ISPL}, based on which the computational efficiency of classical LMS algorithms was significantly improved. Similarly, a modified LMS algorithm was designed by switching difference order \cite{Cheng:2017ISA}. With the help of the series representation, \cite{Wei:2017FCAAa} and \cite{Wei:2019CNSNS} explored the properties of the discrete fractional calculus. \cite{Baleanu:2017CNSNS,Wu:2017AMC,Wei:2018ISA} investigated the stability issue via direct Lyapunov method. Nevertheless, several critical issues are still yet to be addressed: i) those previous results cannot be directly used in the general case that the difference order $\alpha\in(n-1,n),~n\in\mathbb{Z}_+$, especially with nonzero initial instant; ii) several dynamic characteristics of this solution, like convergence, monotonicity and overshoot, still need to be investigated systematically; and iii) the fundamental properties of the nabla discrete Laplace transform and the discrete Mittag--Leffler function need to be explored. To the authors' best knowledge, few papers have been published on solving these problems, which directly motivates this contribution.

The rest of this paper is structured as follows. Section \ref{Section 2} provides some basic facts on this work. In Section \ref{Section 3}, we analyze the time-domain response of a class of nabla discrete fractional order systems in the framework of $N$-transform. An explicit response characterized by discrete Mittag--Leffler functions is developed. In addition, the performance analysis of the system under zero input is conducted. The numerical simulation is performed in Section \ref{Section 4}. Finally, conclusions in Section \ref{Section 5} close the paper.

\section{Preliminaries}\label{Section 2}
In this section, some basic definitions and tools for nabla discrete fractional calculus are reformulated from \cite{Atici:2009EJQTDE}.

The $\alpha$-th nabla fractional sum of a function $f:  \mathbb{N}_{a+1}\to\mathbb{R}$ is defined by
\begin{equation}\label{Eq1}
{\textstyle { {}_a^{}\nabla_k^{-\alpha} }f\left( k \right)  \triangleq \sum\nolimits_{j = 0}^{k-a-1} {{{\left( { - 1} \right)}^{j}}\big( {\begin{smallmatrix}
{ - \alpha }\\
j
\end{smallmatrix}} \big)f\left( {k - j}\right)}},
\end{equation}
where $\alpha>0$, $k\in \mathbb{N}_{a+1}$, $\mathbb{N}_{a+1}\triangleq\{a+1,a+2,a+3,\cdots\}$, ${\left( {\begin{smallmatrix}
{ p }\\
q
\end{smallmatrix}} \right)} \triangleq \frac{{\Gamma \left( {p+1  } \right)}}{{\Gamma \left( {q+1} \right)\Gamma \left( {p-q+1} \right)}}$ and $\Gamma\left(\cdot\right)$ is the Gamma function.

The generalized $N$-transform of a function $f: \mathbb{N}_{a+1}\to \mathbb{R}$ is defined by
\begin{equation}\label{Eq2}
{\textstyle
F\left( s \right)={\mathscr N}_a\left\{ {f\left( k \right)} \right\} \triangleq \sum\nolimits_{k = 1}^{+\infty}  {{{\left( {1 - s} \right)}^{k - 1}}f\left( k+a \right)},~s\in\mathbb{C}},
\end{equation}
which is also called the nabla discrete Laplace transform. The inverse nabla Laplace transform can be obtained as $f\left( k \right)={\mathscr N}_a^{-1}\left\{ {F\left( s \right)} \right\}  \triangleq  \frac{1}{{2\pi {\rm{j}}}}\oint_c {F\left( s \right){{(1 - s)}^{ - k + a}}{\rm{d}}s} ,k \in {\mathbb{N}_{a + 1}}$, where $c$ is a closed curve rotating around the point $(1,{\rm j}0)$ clockwise and it also locates in the convergent domain of $F(s)$.


The following nabla Caputo fractional difference is used in this work
\begin{equation}\label{Eq3}
{}_a^{}\nabla_k ^\alpha f\left( k \right) \triangleq {}_a^{}\nabla_k ^{ \alpha-n }{\nabla ^n}f\left( k \right),
\end{equation}
where $n-1< \alpha<n$, $n\in\mathbb{Z}_+$, $f:\mathbb{N}_{a-n} \to \mathbb{R}$ and ${\nabla ^n}$ represents the normal $n$-th backward difference operation

\begin{equation}\label{Eq4}
{\textstyle
{\nabla ^n}f\left( k \right) \triangleq \sum\nolimits_{j = 0}^n {{{{\left(\hspace{-1pt}{ - 1} \hspace{-1pt}\right)}^j}\left({\begin{smallmatrix}
{ n }\\
j
\end{smallmatrix}}\right)}f\left( {k - j} \right)}.}
\end{equation}

By using the definition (\ref{Eq3}), one obtains that the fractional difference of a constant equals to zero, which coincides with the traditional integer difference case ($\alpha=1$). More importantly, with this definition, the $N$-transform of a fractional  difference has several integer order differences as its initial conditions, which makes $N$-transform more convenient to calculate. These are the two main reasons why (\ref{Eq3}) is usually used to define a nabla fractional difference. In this work, this definition is adopted again.

\section{Main Results}\label{Section 3}
This section explores several useful properties of the nabla Laplace transform and the discrete Mittag--Leffler function. By using these properties, the time--domain response of a class of discrete fractional order systems is then analyzed.

\subsection{Some useful properties}

The following theorem is the initial value and final value theorem of the nabla discrete Laplace transform, corresponding to those of classical $Z$-transform \cite{Cheng:2011Book}.

\begin{theorem}\label{Theorem 1}
Let $f: \mathbb{N}_{a+1}\to \mathbb{R}$ be a bounded function and suppose $f(+\infty)$ exist. Then the following equalities hold:
\begin{enumerate}[i)]
  \item  $f\left( a+1 \right) = \mathop {\lim }\limits_{s \to 1} F\left( s \right)$; and
  \item $ f\left( a+\infty\right) = \mathop {\lim }\limits_{s \to 0} sF\left( s \right)$,
\end{enumerate}
where $F(s)$ is defined in (\ref{Eq2}).
\end{theorem}
\begin{proof}
i) A direct mathematical derivation on (\ref{Eq2}) yields
\begin{eqnarray}\label{Eq5}
\begin{array}{l}
F\left( s \right) = f\left( a+1 \right)+{\left( {1 - s} \right)} f\left( a+2 \right)+{\left( {1 - s} \right)}^2 f\left( a+3 \right)+\cdots.
\end{array}
\end{eqnarray}

Considering the boundedness of $f(n)$, i.e., there exists a constant $M$ satisfying $|f(n)|\le M$ for any $n\in \mathbb{N}$, it follows
\begin{eqnarray}\label{Eq6}
\begin{array}{rl}
|F\left( s \right)-f\left( a+1 \right)|=&\hspace{-6pt}|{\left( {1 - s} \right)} f\left( a+2 \right)+{\left( {1 - s} \right)}^2 f\left( a+3 \right)+\cdots|\\
\le&\hspace{-6pt}|{\left( {1 - s} \right)} f\left( a+2 \right)|+|{\left( {1 - s} \right)}^2 f\left( a+3 \right)|+\cdots\\
\le&\hspace{-6pt} M\big(| 1 - s |+|1 - s |^2+\cdots\big)\\
\le&\hspace{-6pt} M\frac{| 1 - s |}{1-|1-s|}.
\end{array}
\end{eqnarray}
The statement i) can be obtained letting $s\to 1$ on the right-hand side of (\ref{Eq6}).

ii) With the help of
\begin{eqnarray}\label{Eq7}
 {\textstyle {\mathscr N}_a\big\{ {{\nabla ^n }f\left( k \right)} \big\} = {s^n }{\mathscr N}_a\left\{ {f\left( k \right)} \right\} - \sum\nolimits_{j = 0}^{n - 1} {{s^{n  - j - 1}}} {\nabla ^j} {f\left( a \right)},n\in\mathbb{N}_+,}
\end{eqnarray}
one has
\begin{equation}\label{Eq8}
\begin{array}{rl}
{\mathscr N}_a\big\{ {\nabla^1 f\left( k \right)} \big\} =&\hspace{-6pt} \sum\nolimits_{k = 1}^{+ \infty}  {{{\left( {1 - s} \right)}^{k - 1}}\nabla^1 f\left( k +a\right)} \\
 =&\hspace{-6pt} sF\left( s \right) - f\left( a \right).
\end{array}
\end{equation}

Hence,
  \begin{equation}\label{Eq9}
    \begin{array}{rl}
    \mathop {\lim }\limits_{s \to 0} \sum\nolimits_{k = 1}^{+ \infty}  {{{\left( {1 - s} \right)}^{k - 1}}\nabla^1 f\left( k+a \right)}  =&\hspace{-6pt} \mathop {\lim }\limits_{s \to 0} \left[ {sF\left( s \right) - f\left( a \right)} \right]\\
    =&\hspace{-6pt} \sum\nolimits_{k = 1}^{+ \infty}  {\mathop {\lim }\limits_{s \to 0} {{\left( {1 - s} \right)}^{k - 1}}\nabla^1 f\left( k+a \right)} \\
    =&\hspace{-6pt} \sum\nolimits_{k = 1}^{+ \infty} {\nabla^1 f\left( k+a \right)} \\
    =&\hspace{-6pt}\mathop {\lim }\limits_{k \to  {+ \infty} } \left[ {f\left( k+a \right) - f\left( a \right)} \right],
    \end{array}
    \end{equation}
    from which statement ii) can be concluded immediately.
\end{proof}

Along this way, a general case follows
\begin{equation}\label{Eq10}
{\textstyle f\left( {a + \kappa } \right) = \mathop {\lim }\limits_{s \to 1} \frac{{F\left( s \right) - \sum\nolimits_{k = 1}^{\kappa  - 1} {{{\left( {1 - s} \right)}^{k - 1}}f\left( {a + k} \right)} }}{{{{\left( {1 - s} \right)}^{\kappa  - 1}}}},\kappa\in\mathbb{Z}_+.}
\end{equation}

\begin{theorem}\label{Theorem 2}
Let $f: \mathbb{N}_{a+1}\to \mathbb{R}$ and $F(s)={\mathscr N}_a\{f(k)\}$. $f(k)$ is convergent with respect to $k$ if and only if all the main poles of $F(s)$ satisfy $|s-1|>1$ (see Fig. \ref{Fig1}).
\begin{figure}[!htbp]
\centering
\includegraphics[width=0.65\columnwidth]{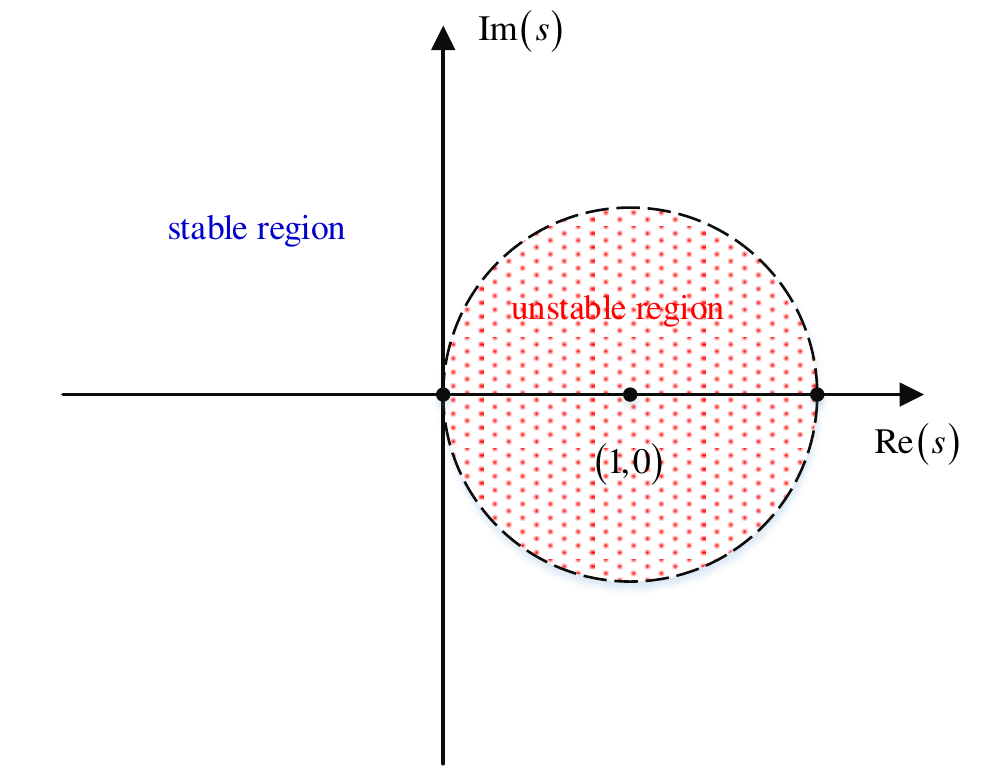}\vspace{-10pt}
\caption{The schematic diagram of stable region in $s$ domain.}\label{Fig1}
\end{figure}
\end{theorem}
\begin{proof} Defining ${\rm{g}}\left( k \right) = f\left( {k + 1} \right)$, $1 - s = {z^{ - 1}}$ and recording $G(z)={\mathscr Z}_a\left\{ {g\left( k \right)} \right\}$, $F(s)={\mathscr N}_a\left\{ {f\left( k \right)} \right\}$, one has
\begin{equation}\label{Eq11}
\begin{array}{rl}
G\left(z\right) =&\hspace{-6pt} \sum\nolimits_{k = 0}^{ + \infty } {{z^{ - k}}g\left( k+a \right)} \\
 =&\hspace{-6pt} \sum\nolimits_{k = 0}^{ + \infty } {{z^{ - k}}f\left( {k +a+1} \right)} \\
 =&\hspace{-6pt} \sum\nolimits_{k = 1}^{ + \infty } {{z^{ - k + 1}}f\left( k+a \right)} \\
 =&\hspace{-6pt} \sum\nolimits_{k = 1}^{ + \infty } {{{\left( {1 - s} \right)}^{k - 1}}f\left( k+a \right)} \\
 =&\hspace{-6pt} F\left( s \right),
\end{array}
\end{equation}
where the $Z$-transform is defined by ${\mathscr Z}_a\left\{ {g\left( k \right)} \right\} \triangleq\sum\nolimits_{k = 0}^{ + \infty } {{z^{ - k}}g\left( k+a \right)}$.

Due to the existence and uniqueness of $Z$-transform, one can claim that the property of $g(k)$ can be summarized by $G(z)$. More specially, $g(k)$ is convergent if and only if all the main poles of $G(z)$ satisfy $|z|<1$. Therefore, the sufficient and necessary condition for the convergence of $f(k)$ is that all the main poles of $F(s)$ satisfy $|1-s|^{-1}<1$ (or $|s-1|>1$). Similarly, the main poles of $F(s)$ reflect the convergence of $f(k)$. All of these complete the proof.
\end{proof}

By referring to the $Z$-transform theory, one has the following remark.
\begin{remark}\label{Remark 1}
Let $f: \mathbb{N}_{a+1}\to \mathbb{R}$ and $F(s)={\mathscr N}_a\{f(k)\}$. When some of the main poles of $F(s)$ satisfy $|s-1|<1$, $f(k)$ is divergent with respect to $k$. When all the main poles of $F(s)$ satisfy $|s-1|=1$, $f(k)$ is constant, oscillating or divergent with respect to $k$.
\end{remark}

If we denote $p^{\overline{q}} = \frac{{\Gamma \left( {p +q } \right)}}{{\Gamma \left( {p} \right)}}$ with $p\in \mathbb{N}$, $q\in\mathbb{R}$, the discrete Mittag--Leffler function can be rewritten as
\begin{eqnarray}\label{Eq12}
{\textstyle
\begin{array}{l}
{{\mathcal F}_{\alpha ,\beta }}\left( {\lambda ,k,a} \right) \triangleq  \sum\nolimits_{j = 0}^{+ \infty}  {\frac{{{\lambda ^j}}}{{\Gamma \left( {j\alpha  + \beta } \right)}} \left( k-a\right)^{\overline{j\alpha  + \beta -1}}},
\end{array}}
\end{eqnarray}
which plays the same role in discrete fractional order systems as the continuous Mittag--Leffler function does in continuous fractional order systems \cite{Gorenflo:2014Book,Li:2018SSRN,Zhang:2015NAHS}. The following properties of ${\mathcal F}_{\alpha,\beta}\left( {\lambda ,k,a} \right) $ (see Theorem \ref{Theorem 4}) will be useful in our analysis later.

\begin{theorem}\label{Theorem 3}
If $n-1<\alpha<n$, $n\in\mathbb{Z}_+$, $\beta>0$, $m\in\mathbb{N}$, $m<\beta$, $\gamma>0$, $\kappa\in\mathbb{N}$, $\kappa<n$ and $k\in\mathbb{N}_a$, then
\begin{enumerate}[i)]
  \item ${\mathscr N}_a\big\{ {{{\mathcal F}_{\alpha ,\beta }}\left( {\lambda ,k,a} \right)} \big\} = \frac{{{s^{\alpha  - \beta }}}}{{{s^\alpha } - \lambda }}$, when $\left|\lambda  \right| < \left| s \right|^\alpha$;
  \item ${\nabla ^m}{{\mathcal F}_{\alpha ,\beta }}\left( {\lambda ,k,a} \right) = {{\mathcal F}_{\alpha ,\beta  - m}}\left( {\lambda ,k,a} \right)$;
  \item ${{}_a^{}\nabla_k ^{-\gamma} }{{\mathcal F}_{\alpha ,\beta }}\left( \lambda ,k,a \right) = {{\mathcal F}_{\alpha  ,\beta+ \gamma }}\left( \lambda ,k,a \right)$; and
  \item ${{}_a^{}\nabla_k ^\alpha }{{\mathcal F}_{\alpha ,\kappa+1 }}\left( {\lambda ,k,a} \right) = \lambda {{\mathcal F}_{\alpha ,\kappa+1 }}\left( {\lambda ,k,a} \right)$.
\end{enumerate}
\end{theorem}
\begin{proof}
i) Applying the $N$-transform to  (\ref{Eq12}) yields
\begin{equation}\label{Eq13}
\begin{array}{rl}
{\mathscr N}_a\big\{ {{{\mathcal F}_{\alpha ,\beta }}\left( {\lambda ,k,a} \right)} \big\}
 =&\hspace{-6pt} \sum\nolimits_{j = 0}^{+ \infty}  {\frac{{{\lambda ^j}}}{{\Gamma \left( {j\alpha  + \beta } \right)}}{\mathscr N}_a\big\{ \left(k-a\right)^{\overline{j\alpha  + \beta -1}}\big\}}\\
 =&\hspace{-6pt} \sum\nolimits_{j = 0}^{ + \infty } {\frac{{{\lambda ^j}}}{{\Gamma \left( {j\alpha  + \beta } \right)}}\sum\nolimits_{k = 1}^{ + \infty } {{{\left( {1 - s} \right)}^{k - 1}}{{\left( {k - a + a} \right)}^{\overline {j\alpha  + \beta  - 1} }}} } \\
 =&\hspace{-6pt} \sum\nolimits_{j = 0}^{ + \infty } {{\lambda ^j}\sum\nolimits_{k = 1}^{ + \infty } {{{\left( {1 - s} \right)}^{k - 1}}\frac{{{k^{\overline {j\alpha  + \beta  - 1} }}}}{{\Gamma \left( {j\alpha  + \beta } \right)}}} } \\
 =&\hspace{-6pt} \sum\nolimits_{j = 0}^{ + \infty } {{\lambda ^j}\sum\nolimits_{k = 1}^{ + \infty } {{{\left( {1 - s} \right)}^{k - 1}}\frac{{\Gamma \left( {k + j\alpha  + \beta  - 1} \right)}}{{\Gamma \left( {j\alpha  + \beta } \right)\Gamma \left( k \right)}}} } \\
 =&\hspace{-6pt} \sum\nolimits_{j = 0}^{ + \infty } {{\lambda ^j}\sum\nolimits_{k = 0}^{ + \infty } {{{\left( {1 - s} \right)}^k}\frac{{\Gamma \left( {k + j\alpha  + \beta } \right)}}{{\Gamma \left( {j\alpha  + \beta } \right)\Gamma \left( {k + 1} \right)}}} } \\
 =&\hspace{-6pt} \sum\nolimits_{j = 0}^{ + \infty } {{\lambda ^j}\sum\nolimits_{k = 0}^{ + \infty } {{{\left( {s - 1} \right)}^k}\frac{{\Gamma \left( {1 - j\alpha  - \beta } \right)}}{{\Gamma \left( {1 - j\alpha  - \beta  - k} \right)\Gamma \left( {k + 1} \right)}}} } \\
 =&\hspace{-6pt} \sum\nolimits_{j = 0}^{ + \infty } {{\lambda ^j}\sum\nolimits_{k = 0}^{ + \infty } {{{\left( {s - 1} \right)}^k}\big( {\begin{smallmatrix}
{ - j\alpha  - \beta }\\
k
\end{smallmatrix}} \big)} } \\
 =&\hspace{-6pt} \sum\nolimits_{j = 0}^{ + \infty } {{\lambda ^j}{{\left( {1 + s - 1} \right)}^{ - j\alpha  - \beta }}} \\
 =&\hspace{-6pt} \sum\nolimits_{j = 0}^{ + \infty } {\frac{{{\lambda ^j}}}{{{s^{j\alpha  + \beta }}}}} \\
 =&\hspace{-6pt} \frac{1}{{{s^\beta }}}\frac{1}{{1 - \frac{\lambda }{{{s^\alpha }}}}}\\
 =&\hspace{-6pt} \frac{{{s^{\alpha  - \beta }}}}{{{s^\alpha } - \lambda }},
\end{array}
\end{equation}
where $\left|\lambda  \right| < \left| s \right|^\alpha$ is adopted.

ii) Applying the backward difference operation (\ref{Eq4}) to the discrete Mittag--Leffler function (\ref{Eq12}) yields
\begin{equation}\label{Eq14}
\begin{array}{rl}
{\nabla ^m}{{\mathcal F}_{\alpha ,\beta }}\left( {\lambda ,k,a} \right) =&\hspace{-6pt} \sum\nolimits_{j = 0}^{ + \infty } {{\lambda ^j}{\nabla ^m}\frac{{{{\left( {k - a} \right)}^{\overline {j\alpha  + \beta  - 1} }}}}{{\Gamma \left( {j\alpha  + \beta } \right)}}} \\
 =&\hspace{-6pt} \sum\nolimits_{j = 0}^{ + \infty } {{\lambda ^j}\sum\nolimits_{i = 0}^m {{{\left( { - 1} \right)}^i}\big(\begin{smallmatrix}
m\\
i
\end{smallmatrix}\big)} \frac{{{{\left( {k - a - i} \right)}^{\overline {j\alpha  + \beta  - 1} }}}}{{\Gamma \left( {j\alpha  + \beta } \right)}}} \\
 =&\hspace{-6pt} \sum\nolimits_{j = 0}^{ + \infty } {{\lambda ^j}\sum\nolimits_{i = 0}^m {{{\left( { - 1} \right)}^i}\big(\begin{smallmatrix}
m\\
i
\end{smallmatrix}\big)} \big(\begin{smallmatrix}
{k - a - i + j\alpha  + \beta  - 2}\\
{k - a - i - 1}
\end{smallmatrix}\big)} \\
 =&\hspace{-6pt} \sum\nolimits_{j = 0}^{ + \infty } {{\lambda ^j}\sum\nolimits_{i = 0}^m {\big(\begin{smallmatrix}
{ - m + i}\\
i
\end{smallmatrix}\big)} \big(\begin{smallmatrix}
{k - a - i + j\alpha  + \beta  - 2}\\
{k - a - i - 1}
\end{smallmatrix}\big)} \\
 =&\hspace{-6pt} \sum\nolimits_{j = 0}^{ + \infty } {{\lambda ^j}\big(\begin{smallmatrix}
{k - a + j\alpha  + \beta  - m - 2}\\
{k - a - 1}
\end{smallmatrix}\big)} \\
 =&\hspace{-6pt} \sum\nolimits_{j = 0}^{ + \infty } {\frac{{{\lambda ^j}}}{{\Gamma \left( {j\alpha  + \beta  - m} \right)}}} {\left( {k - a} \right)^{\overline {j\alpha  + \beta  - m - 1} }}\\
 =&\hspace{-6pt} {\nabla ^m}{{\mathcal F}_{\alpha ,\beta  - m}}\left( {\lambda ,k,a} \right).
\end{array}
\end{equation}

iii) With the help of (\ref{Eq1}), one has
\begin{equation}\label{Eq15}
\begin{array}{rl}
{}_a^{}\nabla _k^{ - \gamma }{{\mathcal F}_{\alpha ,\beta }}\left( {\lambda ,k,a} \right) =&\hspace{-6pt} \sum\nolimits_{j = 0}^{ + \infty } {{\lambda ^j}_a^{}\nabla _k^{ - \gamma }\frac{{{{\left( {k - a} \right)}^{\overline {j\alpha  + \beta  - 1} }}}}{{\Gamma \left( {j\alpha  + \beta } \right)}}} \\
 =&\hspace{-6pt} \sum\nolimits_{j = 0}^{ + \infty } {{\lambda ^j}\sum\nolimits_{i = 0}^{n - a - 1} {{{\left( { - 1} \right)}^i}\big(\begin{smallmatrix}
{ - \gamma }\\
i
\end{smallmatrix}\big)\frac{{{{\left( {k - i - a} \right)}^{\overline {j\alpha  + \beta  - 1} }}}}{{\Gamma \left( {j\alpha  + \beta } \right)}}} } \\
 =&\hspace{-6pt} \sum\nolimits_{j = 0}^{ + \infty } {{\lambda ^j}\sum\nolimits_{i = 0}^{n - a - 1} {{{\left( { - 1} \right)}^i}\big(\begin{smallmatrix}
{ - \gamma }\\
i
\end{smallmatrix}\big)\frac{{\Gamma \left( {k - i - a + j\alpha  + \beta  - 1} \right)}}{{\Gamma \left( {j\alpha  + \beta } \right)\Gamma \left( {k - i - a} \right)}}} } \\
 =&\hspace{-6pt} \sum\nolimits_{j = 0}^{ + \infty } {{\lambda ^j}\sum\nolimits_{i = 0}^{n - a - 1} {{{\left( { - 1} \right)}^{k - a - 1}}\big(\begin{smallmatrix}
{ - \gamma }\\
i
\end{smallmatrix}\big)\frac{{\Gamma \left( { - j\alpha  - \beta  + 1} \right)}}{{\Gamma \left( { - k + i + a - j\alpha  - \beta  + 2} \right)\Gamma \left( {k - i - a} \right)}}} } \\
 =&\hspace{-6pt} \sum\nolimits_{j = 0}^{ + \infty } {{\lambda ^j}{{\left( { - 1} \right)}^{k - a - 1}}\sum\nolimits_{i = 0}^{k - a - 1} {\big(\begin{smallmatrix}
{ - \gamma }\\
i
\end{smallmatrix}\big)\big(\begin{smallmatrix}
{ - j\alpha  - \beta }\\
{k - i - a - 1}
\end{smallmatrix}\big)} } \\
 =&\hspace{-6pt} \sum\nolimits_{j = 0}^{ + \infty } {{\lambda ^j}{{\left( { - 1} \right)}^{k - a - 1}}\big(\begin{smallmatrix}
{ - j\alpha  - \beta  - \gamma }\\
{k - a - 1}
\end{smallmatrix}\big)} \\
 =&\hspace{-6pt} \sum\nolimits_{j = 0}^{ + \infty } {{\lambda ^j}\frac{{{{\left( {k - a} \right)}^{\overline {j\alpha  + \beta  + \gamma  - 1} }}}}{{\Gamma \left( {j\alpha  + \beta  + \gamma } \right)}}} \\
 =&\hspace{-6pt} {{\mathcal F}_{\alpha ,\beta  + \gamma }}\left( {\lambda ,k,a} \right).
\end{array}
\end{equation}

iv) By using (\ref{Eq3}) and (\ref{Eq12}), it obtains
\begin{equation}\label{Eq16}
\begin{array}{rl}
{{}_a^{}\nabla_k ^\alpha }{{\mathcal F}_{\alpha ,\kappa+1}}\left( {\lambda ,k,a} \right) =&\hspace{-6pt} {{}_a^{}\nabla_k ^{ \alpha-n }}{\nabla ^n}{{\mathcal F}_{\alpha ,\kappa+1}}\left( {\lambda ,k,a} \right)\\
 =&\hspace{-6pt} \sum\nolimits_{j = 0}^{+ \infty}  {{\lambda ^j}{{}_a^{}\nabla_k ^{ \alpha-n }}{\nabla ^n}{\frac{{{{\left( {k - a} \right)}^{\overline {j\alpha  + \kappa} }}}}{{\Gamma \left( {j\alpha  + \kappa +1} \right)}}}}.
\end{array}
\end{equation}

Since $\kappa\in\mathbb{N}$ and $\kappa<n$, one has
\begin{equation}\label{Eq17}
{\textstyle
\begin{array}{rl}
{\nabla ^n}{\frac{{{{\left( {k - a} \right)}^{\overline {\kappa  } }}}}{{\Gamma \left( {\kappa+1 } \right)}}}
= &\hspace{-6pt} {\nabla ^{n-\kappa}}{\nabla ^{\kappa}}{\frac{{{{\left( {k - a} \right)}^{\overline {\kappa  } }}}}{{\Gamma \left( {\kappa+1 } \right)}}}\\
=&\hspace{-6pt}{\nabla ^{n-\kappa}}1 = 0.
\end{array}}
\end{equation}
Furthermore, it follows that
\begin{equation}\label{Eq18}
\begin{array}{rl}
{{}_a^{}\nabla_k ^\alpha }{{\mathcal F}_{\alpha ,\kappa}}\left( {\lambda ,k,a} \right) =&\hspace{-6pt} \sum\nolimits_{j = 1}^{+ \infty}  {{\lambda ^j}{{}_a^{}\nabla_k ^{ \alpha-n }} {\frac{{{{\left( {k - a} \right)}^{\overline {j\alpha -n + \kappa} }}}}{{\Gamma \left( {j\alpha -n + \kappa +1} \right)}}}} \\
 =&\hspace{-6pt} \sum\nolimits_{j = 1}^{+ \infty}  {{\lambda ^j} {\frac{{{{\left( {k - a} \right)}^{\overline {j\alpha-\alpha  + \kappa} }}}}{{\Gamma \left( {j\alpha -\alpha + \kappa +1} \right)}}}}\\
 =&\hspace{-6pt} \lambda\sum\nolimits_{j = 0}^{+ \infty}  {{\lambda ^j} {\frac{{{{\left( {k - a} \right)}^{\overline {j\alpha  + \kappa} }}}}{{\Gamma \left( {j\alpha + \kappa +1} \right)}}}} \\
 =&\hspace{-6pt} \lambda {{\mathcal F}_{\alpha ,\kappa}}\left( {\lambda ,k,a} \right),
\end{array}
\end{equation}
which is exactly the statement iv).
\end{proof}

\subsection{Time-domain response analysis}
In this subsection, we will use the properties obtained in Section 3.1 to analyze the time-domain response of a nabla discrete fractional order system. To be specific, we first present an explicit form equipped by discrete Mittag-Leffler functions of this response (see Theorem \ref{Theorem 3}). Then we analyze the time-domain response under zero input condition systematically and rigorously (see Theorem \ref{Theorem 4}).

\begin{theorem}\label{Theorem 4}
If $n - 1 < \alpha  < n, ~n\in\mathbb{Z}_+$, then time--domain response of the discrete fractional order system
\begin{equation}\label{Eq19}
{{}_a^{}{\nabla}_k^\alpha }y\left( k \right)  = \lambda y\left( k \right)+u\left(k \right),~\lambda\neq 1,
\end{equation}
with $y(k)\in \mathbb{R}$, $u(k)\in \mathbb{R}$, and initial conditions
\begin{equation}\label{Eq20}
{\left. {{\nabla}^\kappa }y\left( k \right)  \right|_{k = a}} = {b_\kappa},~\kappa = 0,1, \cdots ,n - 1,
\end{equation}
is unique and is given by
\begin{equation}\label{Eq21}
\begin{array}{rl}
y\left( k \right) =&\hspace{-6pt}\sum\nolimits_{\kappa  = 0}^{n - 1} {{b_\kappa }{{\mathcal F}_{\alpha ,\kappa  + 1}}\left( {\lambda ,k,a} \right)}  + \sum\nolimits_{\tau  = a + 1}^k {{{\mathcal F}_{\alpha ,\alpha }}\left( {\lambda ,\tau ,a} \right)u\left( {k - \tau  + a + 1} \right)}.
\end{array}
\end{equation}
\end{theorem}
\begin{proof}
Reformulating (\ref{Eq19}) gives an equivalent fractional order sum equation as
\begin{eqnarray}\label{Eq22}
{\textstyle
\begin{array}{rl}
{}_a^{}\nabla _k^{ - \alpha }{}_a^{}\nabla _k^\alpha y\left( k \right) =&\hspace{-6pt} y\left( k \right) - \sum\nolimits_{\kappa  = 0}^{n - 1} {\frac{{{b_\kappa }}}{{\Gamma \left( {\kappa  + 1} \right)}}{{\left( {k - a} \right)}^{\overline \kappa  }}} \\
 =&\hspace{-6pt} {\lambda}{}_a^{}\nabla _k^{ - \alpha }y\left( k \right){ + }{}_a^{}\nabla _k^{ - \alpha }u\left( k \right).
\end{array}}
\end{eqnarray}

Define ${y_0}\left( k \right) = \sum\nolimits_{\kappa  = 0}^{n - 1} {\frac{{{b_\kappa }}}{{\Gamma \left( {\kappa  + 1} \right)}}{{\left( {k - a} \right)}^{\overline \kappa  }}}$ and introduce a series of functions ${y_p}\left( k \right),~p \in \mathbb{Z}_+$ by the following relationship
\begin{equation}\label{Eq23}
{\textstyle
{y_{p}}\left( k\right) = {y_0}\left( k \right) + {\lambda}{}_a^{}\nabla _k^{ - \alpha }y_{p-1}\left( k \right){ + }{}_a^{}\nabla _k^{ - \alpha }u\left( k \right).}
\end{equation}

Following formulas (\ref{Eq22}) and (\ref{Eq23}), one obtains
\begin{eqnarray}\label{Eq24}
{\textstyle
{y_p}\left( k \right)\hspace{-1pt} =\sum\nolimits_{j = 0}^{p} {{\lambda ^j}{{}_a^{}\nabla _k^{ - j\alpha }}{y_0}\left( k \right)}  + \sum\nolimits_{j = 0}^{p - 1} {{\lambda ^j}{{}_a^{}\nabla _k^{ - (j+1)\alpha }}u\left( k \right)} .}
\end{eqnarray}

Taking the limit of (\ref{Eq24}) and recalling the definition of the fractional order sum (\ref{Eq1}), yields
\begin{eqnarray}\label{Eq25}
\begin{array}{rl}
y\left( k \right) =&\hspace{-6pt} \mathop {\lim }\limits_{p \to  + \infty } {y_p}\left( k \right)\\
 =&\hspace{-6pt} \sum\nolimits_{j = 0}^{ + \infty } {{\lambda ^j}_a^{}\nabla _k^{ - j\alpha }{y_0}\left( k \right)}  + \sum\nolimits_{j = 0}^{ + \infty } {{\lambda ^j}_a^{}\nabla _k^{ - (j + 1)\alpha }u\left( k \right)} \\
 =&\hspace{-6pt} \sum\nolimits_{j = 0}^{ + \infty } {{\lambda ^j}\sum\nolimits_{\kappa  = 0}^{n - 1} {\frac{{{b_\kappa }}}{{\Gamma \left( {j\alpha  + \kappa  + 1} \right)}}{{\left( {k - a} \right)}^{\overline {j\alpha  + \kappa } }}} }  + \sum\nolimits_{j = 0}^{ + \infty } {{\lambda ^j}\sum\nolimits_{\tau  = 0}^{k - a - 1} {\frac{{{{\left( {\tau  + 1} \right)}^{\overline {j\alpha  + \alpha  - 1} }}}}{{\Gamma \left( {j\alpha  + \alpha } \right)}}u\left( {k - \tau } \right)} } \\
 =&\hspace{-6pt} \sum\nolimits_{\kappa  = 0}^{n - 1} {{b_\kappa }\sum\nolimits_{j = 0}^{ + \infty } {\frac{{{\lambda ^j}}}{{\Gamma \left( {j\alpha  + \kappa  + 1} \right)}}} } {\left( {k - a} \right)^{\overline {j\alpha  + \kappa } }} + \sum\nolimits_{\tau  = a + 1}^k {u\left( {k - \tau  + a + 1} \right)\sum\nolimits_{j = 0}^{ + \infty } {\frac{{{\lambda ^j}}}{{\Gamma \left( {j\alpha  + \alpha } \right)}}{{\left( {\tau  - a} \right)}^{\overline {j\alpha  + \alpha  - 1} }}} } \\
 =&\hspace{-6pt} \sum\nolimits_{\kappa  = 0}^{n - 1} {{b_\kappa }{{\mathcal F}_{\alpha ,\kappa  + 1}}\left( {\lambda ,k,a} \right)}  + \sum\nolimits_{\tau  = a + 1}^k {{{\mathcal F}_{\alpha ,\alpha }}\left( {\lambda ,\tau ,a} \right)u\left( {k - \tau  + a + 1} \right)}.
\end{array}
\end{eqnarray}
By taking fractional difference to $y(k)$ in (\ref{Eq25}), one has
\begin{equation}\label{Eq26}
{\textstyle \begin{array}{rl}
{}_a^{}\nabla _k^\alpha y\left( k \right) =&\hspace{-6pt} {}_a^{}\nabla _k^\alpha \big[ {\sum\nolimits_{j = 0}^{ + \infty } {{\lambda ^j}{}_a^{}\nabla _k^{ - j\alpha }{y_0}\left( k \right)}  + \sum\nolimits_{j = 0}^{ + \infty } {{\lambda ^j}_a^{}\nabla _k^{ - (j + 1)\alpha }u\left( k \right)} } \big]\\
 =&\hspace{-6pt} \sum\nolimits_{j = 0}^{ + \infty } {{\lambda ^j}{}_a^{}\nabla _k^{ - j\alpha  + \alpha }{y_0}\left( k \right)}  + \sum\nolimits_{j = 0}^{ + \infty } {{\lambda ^j}{}_a^{}\nabla _k^{ - j\alpha }u\left( k \right)} \\
 =&\hspace{-6pt} \lambda \sum\nolimits_{j =  - 1}^{ + \infty } {{\lambda ^j}{}_a^{}\nabla _k^{ - j\alpha }{y_0}\left( k \right)}  + \lambda \sum\nolimits_{j =  - 1}^{ + \infty } {{\lambda ^j}{}_a^{}\nabla _k^{ - (j + 1)\alpha }u\left( k \right)} \\
 =&\hspace{-6pt} \lambda \big[ {\sum\nolimits_{j = 0}^{ + \infty } {{\lambda ^j}{}_a^{}\nabla _k^{ - j\alpha }{y_0}\left( k \right)}  + \sum\nolimits_{j = 0}^{ + \infty } {{\lambda ^j}{}_a^{}\nabla _k^{ - (j + 1)\alpha }u\left( k \right)} } \big] + {}_a^{}\nabla _k^\alpha {y_0}\left( k \right) + {}_a^{}\nabla _k^0u\left( k \right)\\
 =&\hspace{-6pt} \lambda y\left( k \right) + u\left( k \right),
\end{array}}
\end{equation}
which is equal to equation (\ref{Eq19}).

Now we prove the uniqueness of the response in (\ref{Eq21}). Suppose that $\bar y(k)$ is another response of the system (\ref{Eq19}) under the same initial conditions (\ref{Eq20}). Defining the function $\epsilon(k)=y(k)-\bar y(k)$, one has that ${}_a^{}\nabla _k^\alpha\epsilon\left( k \right)=\lambda\epsilon\left( k \right)$ with initial conditions ${\left. {{\nabla}^\kappa }\epsilon\left( k \right)  \right|_{k = a}} = 0,~\kappa = 0,1, \cdots ,n - 1$. Performing the fractional order sum operation, an equivalent fractional order sum equation follows
\begin{equation}\label{Eq27}
{\textstyle \begin{array}{rl}
\epsilon \left( k \right) =&\hspace{-6pt} \lambda {{}_a^{}\nabla_k ^{ - \alpha }}\epsilon \left( k \right)\\
 =&\hspace{-6pt} \frac{\lambda}{{\Gamma \left( \alpha  \right)}}\sum\nolimits_{j = 0}^{k-a-1} {\frac{{\Gamma \left( {j + \alpha} \right)}}{{\Gamma \left( j+1 \right)}}\epsilon \left( {k - j } \right).}
\end{array}}
\end{equation}

When $k=a+1$, (\ref{Eq27}) becomes $\epsilon \left( a+1 \right) =\lambda\epsilon \left( a+1 \right)$. Recalling the condition $\lambda\neq1$, one has that $\epsilon \left( a+1 \right)=0$. When $k=a+2$, one has that $\epsilon \left( a+2 \right) =\lambda\epsilon \left( a+2 \right)$, which implies $\epsilon \left( a+2 \right)=0$. As an analogy, one can conclude that $\epsilon \left( a+j \right)\equiv0,~j=1,2,\cdots$. The uniqueness of a non-zero solution is thus proved.
\end{proof}


With the obtained Theorem \ref{Theorem 4}, we are ready to analyze the zero input response of the discrete fractional order system in formula (\ref{Eq19}). 
\begin{theorem}\label{Theorem 5}
If $n-1<\alpha<n,~n\in \mathbb{Z}_+$, then the response to the fractional order system (\ref{Eq19}) with $u(k)\equiv 0$ has the following properties:
\begin{enumerate}[i)]
  \item For the case of $\lambda>0$ and $\lambda\neq1$, if $\lambda  < {2^\alpha }$, $y(k)$ is divergent; and if $\lambda  > {2^\alpha }$, $y(k)$ is convergent.
  \item For the case of $\lambda=0$, $y\left( k \right) = \sum\nolimits_{\kappa = 0}^{n - 1} {\frac{{{b_\kappa}}}{{\kappa!}}} {{(k-a)^{\overline{\kappa}}}}$ and therefore $y(k)$ might be divergent or constant.
  \item For the case of $\lambda<0$, if $0<\alpha\le1$, $y(k)$ is monotonically convergent; If $1<\alpha\le2$, $y(k)$ is convergent with a possible overshoot; If $\alpha>2$, three situations will be divided. More specially, when $\left| \lambda  \right| < 2^{\alpha } \cos ^{\alpha } \big( {\frac{\pi }{\alpha }} \big)$, $y(k)$ is divergent; and when $\left| \lambda  \right|> 2^{\alpha }\cos^{\alpha } \big( {\frac{\pi }{\alpha }} \big)$, $y(k)$ is convergent.
\end{enumerate}
\end{theorem}
\begin{proof}
With the initial conditions in (\ref{Eq20}), the $N$-transform of $y(k)$ can be obtained as
\begin{equation}\label{Eq28}
{\textstyle Y\left( s \right) = \sum\nolimits_{\kappa  = 0}^{n - 1} {\frac{{{s^{\alpha  - \kappa  - 1}}}}{{{s^\alpha } - \lambda }}{b_\kappa }} } ,
\end{equation}
whose pole is ${\lambda ^{\frac{1}{\alpha }}}$.

i) When $\lambda>0$ and $\lambda\neq0$, one has ${\lambda ^{\frac{1}{\alpha }}} > 0$ for any positive $\alpha$. If $\lambda  < {2^\alpha }$, it follows ${\lambda ^{\frac{1}{\alpha }}} <2$ and then $|{\lambda ^{\frac{1}{\alpha }}}-1|<1$. By applying Remark \ref{Remark 1}, the divergence of $y(k)$ can be obtained. Similarly, if $\lambda  > {2^\alpha }$, the convergence of $y(k)$ can be achieved from Theorem \ref{Theorem 2}.

ii) When $\lambda=0$ and $u(k)\equiv0$, the response $y(k)$ in (\ref{Eq21}) becomes
  \begin{equation}\label{Eq29}
    \begin{array}{l}
    y\left( k \right)= \sum\nolimits_{\kappa = 0}^{n - 1} {\frac{{{b_\kappa}}}{{\kappa!}}} (k-a)^{\overline{\kappa}},
    \end{array}
  \end{equation}
from which the statement ii) can be directly concluded, since $(k-a)^{\overline{\kappa}}$ is positive and divergent for any $\kappa=1,2,\cdots,n-1$ and $(k-a)^{\overline{0}}=1$. In other words, if $0<\alpha<1$, $y(k)$ is constant. If $\alpha>1$ and $b_\kappa=0,\kappa>1$, $y(k)$ is also constant. If $\alpha>1$ and not all of $b_\kappa$ are zero, $y(k)$ will be divergent.

iii) Letting $\lambda<0$, one has
  \begin{equation}\label{Eq30}
{\lambda ^{\frac{1}{\alpha }}} = {\left| \lambda  \right|^{\frac{1}{\alpha }}}{{\rm{e}}^{-{\rm{j}}\frac{\pi }{\alpha }}},
  \end{equation}
  whose magnitude is ${\left| \lambda  \right|^{\frac{1}{\alpha }}}$ and phase is $-\frac{\pi }{\alpha }$. When $0<\alpha\le2$, ${\lambda ^{\frac{1}{\alpha }}} $ lies in the left half plane and the formula $|{\lambda ^{\frac{1}{\alpha }}} -1|>1$ holds. As a result, $y(k)$ is convergent in this case. When $\alpha>2$, it is difficult to determine whether $|{\lambda ^{\frac{1}{\alpha }}} -1|$ is larger than $1$. Thus, more details will be provided.

 \begin{itemize}
  \item {\bf The case of $0<\alpha\le1$}
\end{itemize}

From the previous discussion, it is known that when $\lambda<0$ and $0<\alpha\le1$, $y(k)$ will converge to $0$ as $k$ increases. Next, the monotonicity of $y(k)$ will be analyzed. In this case, $\kappa=0$, i.e., $y(k)=b_0{{{\mathcal F}_{\alpha ,1}}\left( {\lambda ,k,a} \right)} $. On one hand, when $\alpha=1$ and $\lambda<0$, one has
 \begin{equation}\label{Eq31}
  \begin{array}{rl}
{{\mathcal F}_{1,1}}\left( {\lambda ,k,a} \right) = &\hspace{-6pt}\sum\nolimits_{j = 0}^{{+ \infty}}  {{\lambda ^j}\frac{{\Gamma \left( {j + k-a} \right)}}{{\Gamma \left( {j + 1} \right)\Gamma \left( k-a\right)}}} \\
 =&\hspace{-6pt} \sum\nolimits_{j = 0}^{{+ \infty}}  {{{\left( { - \lambda } \right)}^j}\frac{{\Gamma \left( {1 - k+a} \right)}}{{\Gamma \left( {j + 1} \right)\Gamma \left( {1 - j - k+a} \right)}}} \\
 =&\hspace{-6pt} \sum\nolimits_{j = 0}^{{+ \infty}}  {{{\left( { - \lambda } \right)}^j}\big(  {\begin{smallmatrix}
{ - k}\\
j
\end{smallmatrix}} \big)} \\
 =&\hspace{-6pt} {\left( {1 - \lambda }\right)^{ - k+a}},
\end{array}
\end{equation}
which is monotonically convergent as $k$ increases.

On the other hand, it has been shown in \cite{Cheng:2011Book} that for any $0<\alpha<1$, $\lambda<0$ and $k\in\mathbb{N}_a$,
\begin{equation}\label{Eq32}
  {{{\mathcal F}_{\alpha ,1}}\left( {\lambda ,k,a} \right)}\ge{{{\mathcal F}_{1 ,1}}\left( {\lambda ,k,a} \right)}.
\end{equation}

Then from the statement iii) of Theorem \ref{Theorem 3}, one has
\begin{equation}\label{Eq33}
  {{}_a^{}\nabla_k ^\alpha }{{\mathcal F}_{\alpha ,1}}\left( {\lambda ,k,a } \right) = \lambda {{\mathcal F}_{\alpha ,1 }}\left( {\lambda ,k,a} \right)<0,
\end{equation}
which implies that ${\nabla ^1 }{{\mathcal F}_{\alpha ,1 }}\left( {\lambda ,k,a} \right)<0$ based on Theorem 3.126 of \cite{Goodrich:2015Book}. Thus, ${{\mathcal F}_{\alpha ,1 }}\left( {\lambda ,k,a} \right)$ is monotonically decreasing when $0<\alpha<1$.

Thus one can conclude that for any $0<\alpha\le1$, $y(k)$ is monotonically convergent as $k$ increases.

\begin{itemize}
  \item {\bf The case of $1<\alpha\le2$}
\end{itemize}

In this case, $y(k)=b_0{{\mathcal F}_{\alpha ,1 }}\left( {\lambda ,k,a} \right)+b_1{{\mathcal F}_{\alpha ,2 }}\left( {\lambda ,k,a} \right) $, which is convergent from the aforementioned facts. Next, the overshoot phenomenon of $y(k)$ will be discussed. Firstly, let us discuss the property of ${{\mathcal F}_{\alpha ,1 }}\left( {\lambda ,k,a} \right)$. By using the $N$-transform, we have
\begin{equation}\label{Eq34}
\begin{array}{rl}
{\mathscr N}_a\left\{ {{}_a^{}\nabla_k^ {-\sigma}{{\mathcal F}_{\alpha ,1  }}\left( {\lambda ,k,a} \right)} \right\}=&\hspace{-6pt} s{\mathscr N}_a\left\{ {{{\mathcal F}_{\alpha ,2 + \sigma }}\left( {\lambda ,k,a} \right)} \right\}\\
 = &\hspace{-6pt}\frac{{{s^{\alpha  - \sigma  - 1}}}}{{{s^\alpha } - \lambda }},
\end{array}
\end{equation}
for any $0<\sigma<\alpha-1$.

Theorem \ref{Theorem 3} indicates that
\begin{equation}\label{Eq35}
{\textstyle{{{}_a^{}\nabla_k^ {-\sigma} }{{\mathcal F}_{\alpha ,1}}\left( {\lambda ,a+1,a} \right)} = \mathop {\lim }\limits_{s \to 1} \frac{{{s^{\alpha  - \sigma  - 1}}}}{{{s^\alpha } - \lambda }} = \frac{1}{{1 - \lambda }}>0,}
\end{equation}
\begin{equation}\label{Eq36}
{\textstyle {{\mathcal F}_{\alpha ,2 + \sigma }}\left( {\lambda ,a+ \infty,a} \right) =\mathop {\lim }\limits_{s \to 0} \frac{{{s^{\alpha  - \sigma  - 1}}}}{{{s^\alpha } - \lambda }} = 0.}
\end{equation}

A further calculation on (\ref{Eq34}) leads to
\begin{equation}\label{Eq37}
\begin{array}{l}
\mathop {\lim }\limits_{s \to 0} {\mathscr N}_a\left\{ {{{}_a^{}\nabla_k^ {-\sigma} }{{\mathcal F}_{\alpha ,1}}\left( {\lambda ,k,a} \right)} \right\} \\
= \mathop {\lim }\limits_{s \to 0} \sum\nolimits_{k = 1}^{{+ \infty}}  {{{\left( {1 - s} \right)}^{k - 1}}{{}_a^{}\nabla_k^ {-\sigma} }{{\mathcal F}_{\alpha ,1}}\left( {\lambda ,k+a,a} \right)}\\
 = \sum\nolimits_{k = 1}^{{+ \infty}}  {\mathop {\lim }\limits_{s \to 0} {{\left( {1 - s} \right)}^{k - 1}}{{}_a^{}\nabla_k^ {-\sigma} }{{\mathcal F}_{\alpha ,1}}\left( {\lambda ,k+a,a} \right)} \\
 = \sum\nolimits_{k = 1}^{{+ \infty}}  {{{}_a^{}\nabla_k^ {-\sigma} }{{\mathcal F}_{\alpha ,1}}\left( {\lambda ,k+a,a} \right)}\\
  = 0,
\end{array}
\end{equation}
which implies that ${{{}_a^{}\nabla_k^ {-\sigma} }{{\mathcal F}_{\alpha ,1}}\left( {\lambda ,k,a} \right)}$ cannot be always positive. Therefore, one can conclude that
${{\mathcal F}_{\alpha ,1}}\left( {\lambda ,k,a} \right)$ should change its sign as $n$ increases based on the hypothesis method. In fact, suppose that ${\mathcal F}_{a,1}(\lambda, k,a)>0$ holds for all $k \in \mathbb{N}_a$, then ${{{}_a^{}\nabla_k^ {-\sigma} }{{\mathcal F}_{\alpha ,1}}\left( {\lambda ,k,a} \right)}>0$ holds for all $k \in \mathbb{N}_a$ and vice versa, which leads to a contradiction.

Secondly, let us consider ${{\mathcal F}_{\alpha ,2}}\left( {\lambda ,k,a} \right)$. When $\left|\lambda  \right| < \left| s \right|^\alpha$, the statement i) of Theorem \ref{Theorem 3} indicates that
\begin{equation}\label{Eq38}
{\textstyle {\mathscr N}_a\left\{ {{{\mathcal F}_{\alpha ,2 }}\left( {\lambda ,k,a} \right)} \right\} = \frac{{{s^{\alpha  - 2 }}}}{{{s^\alpha } - \lambda }}}.
\end{equation}

Based on Theorem \ref{Theorem 1}, one has
\begin{equation}\label{Eq39}
{\textstyle{{\mathcal F}_{\alpha ,2}}\left( {\lambda ,a+1,a} \right) = \mathop {\lim }\limits_{s \to 1} \frac{{{s^{\alpha  - 2}}}}{{{s^\alpha } - \lambda }} = \frac{1}{{1 - \lambda }},}
\end{equation}
\begin{equation}\label{Eq40}
{\textstyle{{\mathcal F}_{\alpha ,2}}\left( {\lambda ,a+2,a} \right) = \mathop {\lim }\limits_{s \to 1}\frac{ \frac{{{s^{\alpha  - 2}}}}{{{s^\alpha } - \lambda }}-\frac{1-s}{1-\lambda}}{1-s} = \frac{1-\alpha}{{1 - \lambda }},}
\end{equation}
\begin{equation}\label{Eq41}
{\textstyle {{\mathcal F}_{\alpha ,2}}\left( {\lambda , a+ \infty,a } \right) = \mathop {\lim }\limits_{s \to 0} s\frac{{{s^{\alpha  - 2}}}}{{{s^\alpha } - \lambda }} = 0}.
\end{equation}

Moreover, according to the property of Gamma function, one has
\begin{equation}\label{Eq42}
{\textstyle \begin{array}{rl}
{{\mathcal F}_{\alpha ,2}}\left( {\lambda ,a,a} \right) =&\hspace{-6pt} \sum\nolimits_{j = 0}^{ + \infty } {{\lambda ^j}\frac{{\Gamma \left( {j\alpha  + 1} \right)}}{{\Gamma \left( {j\alpha  + 2} \right)\Gamma \left( 0 \right)}}} \\
 =&\hspace{-6pt} \sum\nolimits_{j = 0}^{ + \infty } {\frac{{{\lambda ^j}}}{{\left( {j\alpha  + 1} \right)\Gamma \left( 0 \right)}}} \\
 =&\hspace{-6pt}0.
\end{array}}
\end{equation}

From (\ref{Eq39})-(\ref{Eq42}), an overshoot can be observed as ${{\mathcal F}_{\alpha ,2}}\left( {\lambda , k,a } \right)$ converges to zero.

In summary, if one of $b_0$ and $b_1$ equals to zero, $y(k)$ will be convergent with an overshoot. Otherwise, $y(k)$ is convergent while the overshoot may disappear for some special values of variables $b_0$ and $b_1$.
\begin{itemize}
  \item {\bf The case of $\alpha>2$}
\end{itemize}

According to (\ref{Eq30}), it can be found that ${\lambda ^{\frac{1}{\alpha }}}$ locates in the right half plane when $\alpha>2$ and $\lambda<0$. Then one only needs to judge whether the point ${\lambda ^{\frac{1}{\alpha }}}$ is in the unstable circle.
Supposing that point $B$ is on the circle and $\angle AOB = \frac{\pi }{\alpha }$ shown as Fig. \ref{Fig2}, because $A$ is the center of a unit circle, $O$ is a point on the circle, then one has
\begin{equation}\label{Eq43}
{\textstyle \begin{array}{rl}
\left| {OB} \right| =&\hspace{-6pt} 2\left| {OA} \right|\cos \left( {\angle AOB} \right)\\
 =&\hspace{-6pt} 2\cos \big( {\frac{\pi }{\alpha }} \big),
\end{array}}
\end{equation}

\begin{figure}[!htbp]
\centering
\includegraphics[width=0.65\columnwidth]{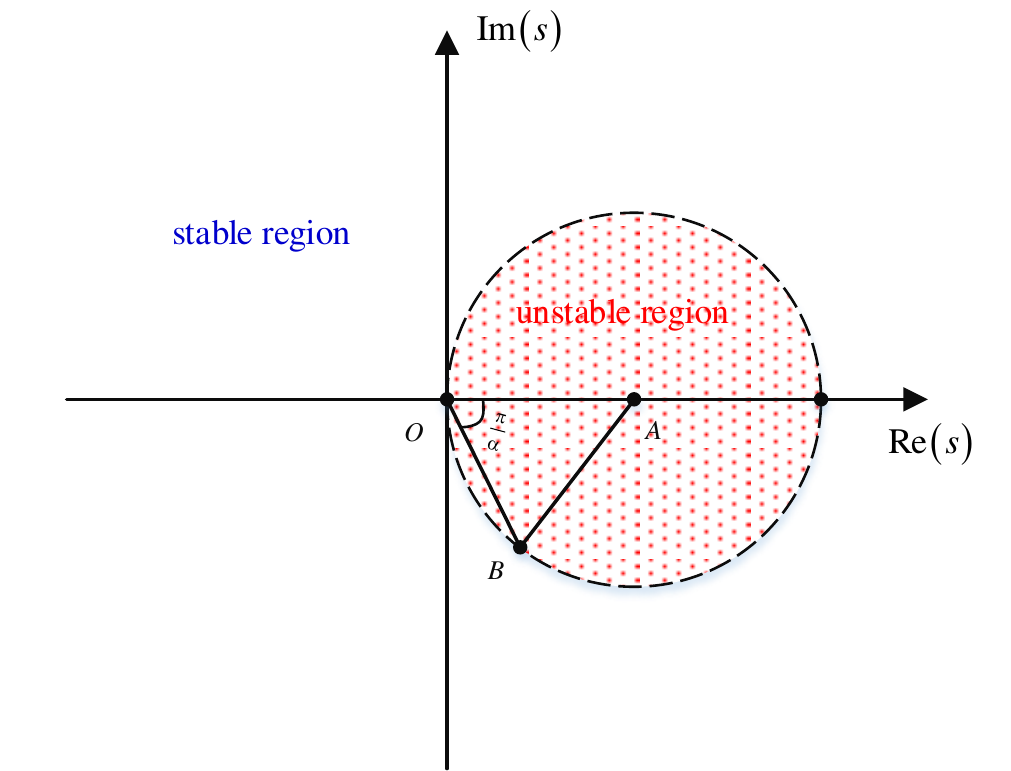}\vspace{-10pt}
\caption{The case of ${\lambda ^{\frac{1}{\alpha }}}$ lying on the circle.}\label{Fig2}
\end{figure}

If ${\left| \lambda  \right|^{\frac{1}{\alpha }}}$ is larger than $\left| {OB} \right|$, ${\lambda ^{\frac{1}{\alpha }}}$ is outside the circle. Similarly, if ${\left| \lambda  \right|^{\frac{1}{\alpha }}}$ is less than $\left| {OB} \right|$, ${\lambda ^{\frac{1}{\alpha }}}$ is in the circle. If ${\left| \lambda  \right|^{\frac{1}{\alpha }}}$ is equal to $\left| {OB} \right|$, ${\lambda ^{\frac{1}{\alpha }}}$ is on the circle. From the property of the exponential function, it can be concluded in a brief form. a) if $\left| \lambda  \right| > 2^{\alpha } \cos ^{\alpha } \big( {\frac{\pi }{\alpha }} \big)$, ${\lambda ^{\frac{1}{\alpha }}}$ is outside the circle; b) if $\left| \lambda  \right| < 2^{\alpha } \cos ^{\alpha } \big( {\frac{\pi }{\alpha }} \big)$, ${\lambda ^{\frac{1}{\alpha }}}$ is in the circle; c) if $\left| \lambda  \right| = 2^{\alpha } \cos ^{\alpha } \big( {\frac{\pi }{\alpha }} \big)$, ${\lambda ^{\frac{1}{\alpha }}}$ is on the circle. Based on the stable theory in Theorem \ref{Theorem 2} and Remark \ref{Remark 1}, the statement iii) in Theorem \ref{Theorem 5} can be derived immediately.
%
\end{proof}

One interesting proposition under critical conditions is proposed here.
\begin{proposition}  For $y\left( k \right) =\sum\nolimits_{\kappa = 0}^{n - 1} {{b_\kappa}{{\mathcal F}_{\alpha ,\kappa + 1}}\left( {\lambda ,k,a } \right)}$ with $n - 1 < \alpha  < n, ~n\in\mathbb{Z}_+$, the following conclusions can be drawn.
\begin{enumerate}[i)]
  \item When $\lambda>0$ and $\lambda  = {2^\alpha }$, if $\alpha\le1$, $y(k)$ is oscillating and if $\alpha>1$, $y(k)$ is convergent.
  \item When $\lambda<0$ and $\left| \lambda  \right| = 2^{\alpha }\cos^{\alpha } \big( {\frac{\pi }{\alpha }} \big)$, $y(k)$ is oscillating.
\end{enumerate}
 \end{proposition}
\begin{remark}\label{Remark 2}
It is worth pointing out that this paper establishes a connection of discrete fractional order systems in time domain and those in frequency domain, based on which some original properties of the $N$-transform and the discrete Mittag--Leffler function are explored. It is believed that the results derived in this note can be conveniently applied to response calculation, stability analysis and controller design for discrete fractional order systems.
\end{remark}

\section{Simulation Study}\label{Section 4}
In this section, without loss of generality, let us set $a=1$.
the zero input time--domain response of system (\ref{Eq19}) in the following four cases:
\[
\left\{\begin{array}{l}
{\rm case}\,1:\,b_0=1,\,\lambda=-0.20,\,\alpha\in\{0.1,0.2,\cdots,1.0\},\\
{\rm case}\,2:\,b_0=1,\,\lambda=-0.20,\,\alpha\in\{1.0,1.1,\cdots,2.0\},\\
{\rm case}\,3:\,b_1=1,\,\lambda=-0.20,\,\alpha\in\{1.1,1.2,\cdots,2.0\},\\
{\rm case}\,4:\,b_0=1,\,\lambda\in\{-0.04,-0.08,\cdots,-0.4\},\,\alpha=1.5,
\end{array}\right.
\]
are shown in Fig. \ref{Fig3} - Fig. \ref{Fig6}, respectively.

\begin{figure}[!htbp]
\centering
\includegraphics[width=0.65\columnwidth]{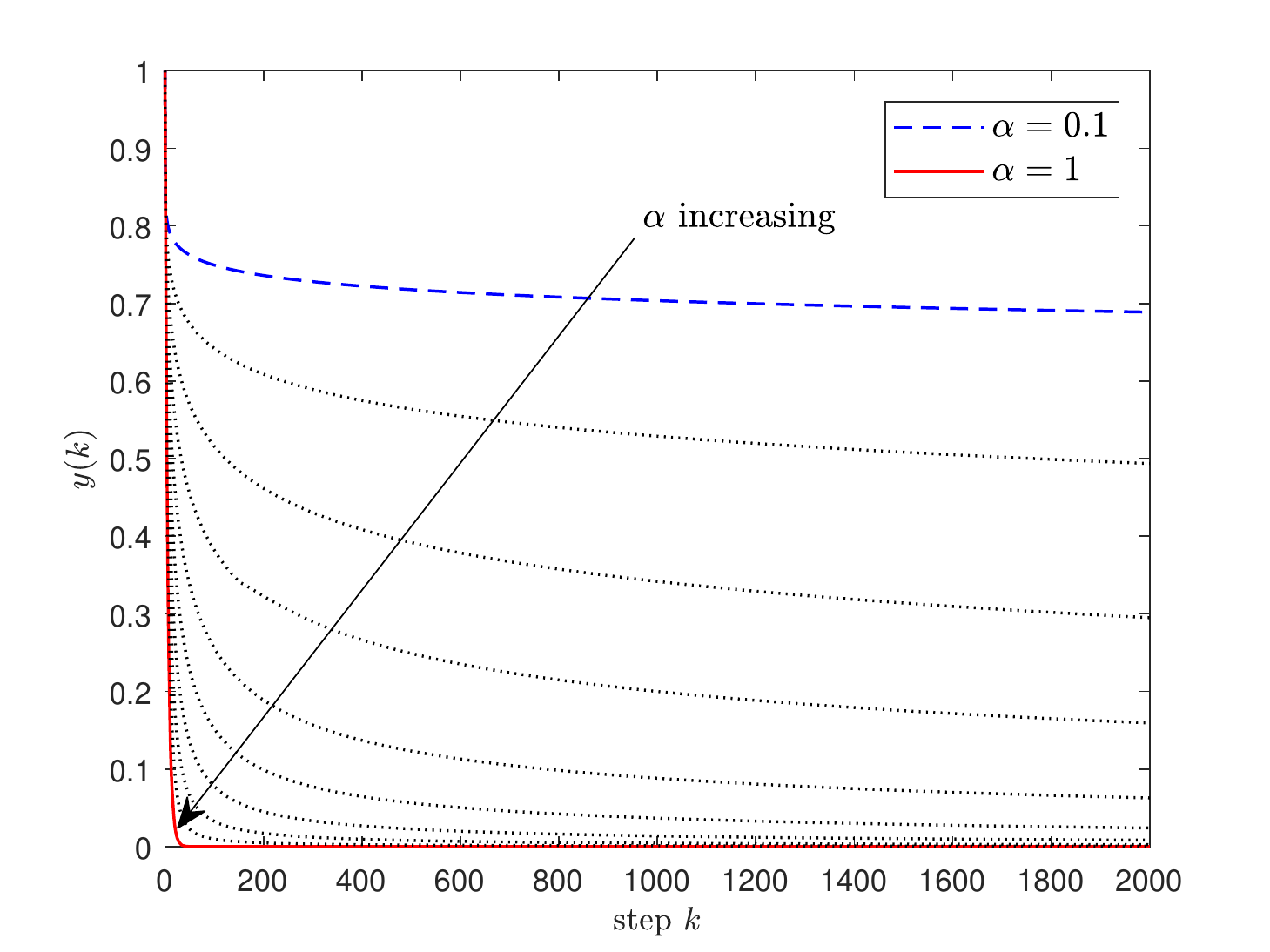}\vspace{-10pt}
\caption{Response $y(k)$ for different $\alpha\in(0,1]$.}\label{Fig3}
\end{figure}
\begin{figure}[!htbp]
\centering
\includegraphics[width=0.65\columnwidth]{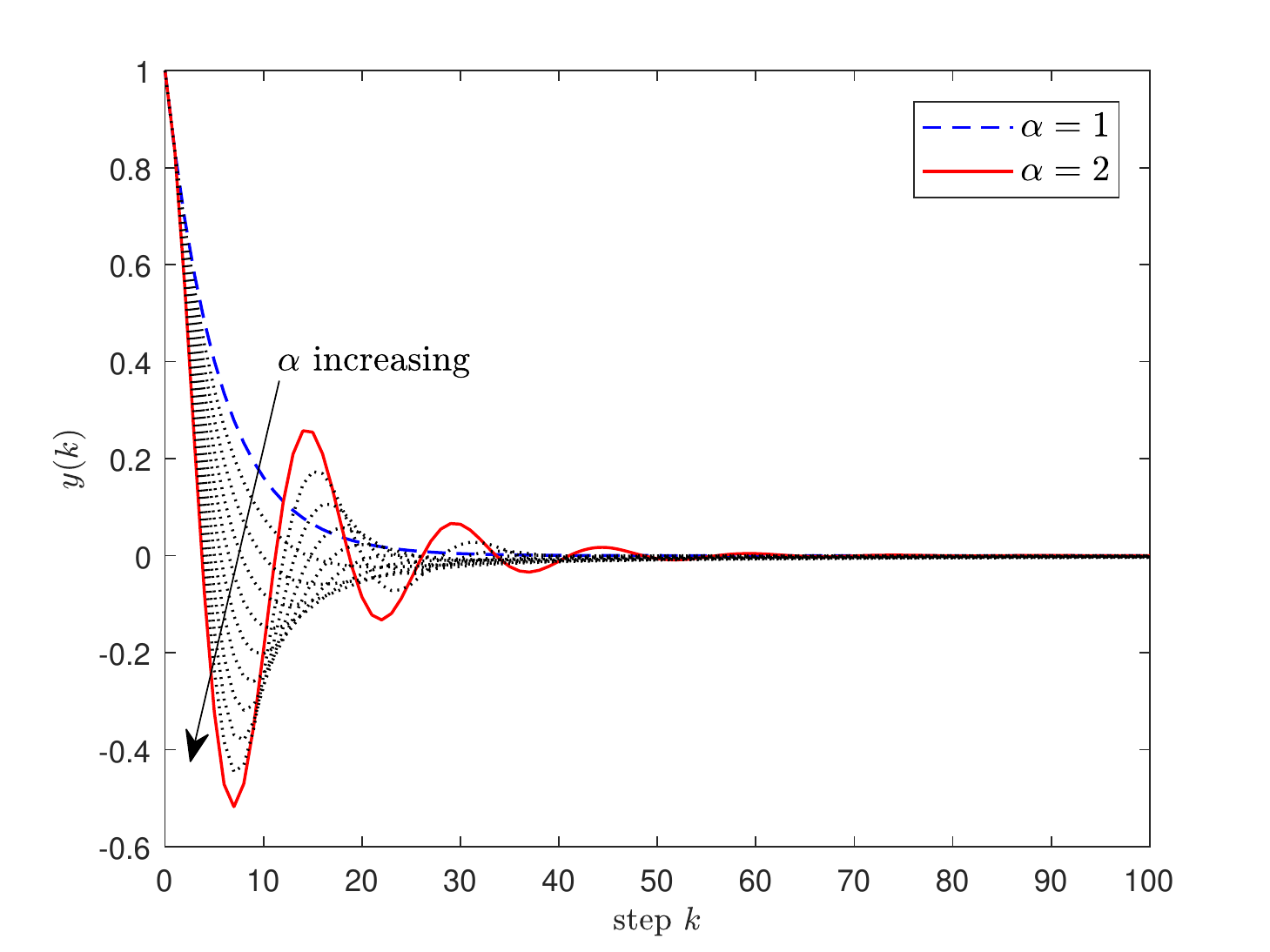}\vspace{-10pt}
\caption{Response $y(k)$ for different $\alpha\in[1,2]$ with $b_1=0$.}\label{Fig4}
\end{figure}
\begin{figure}[!htbp]
\centering
\includegraphics[width=0.65\columnwidth]{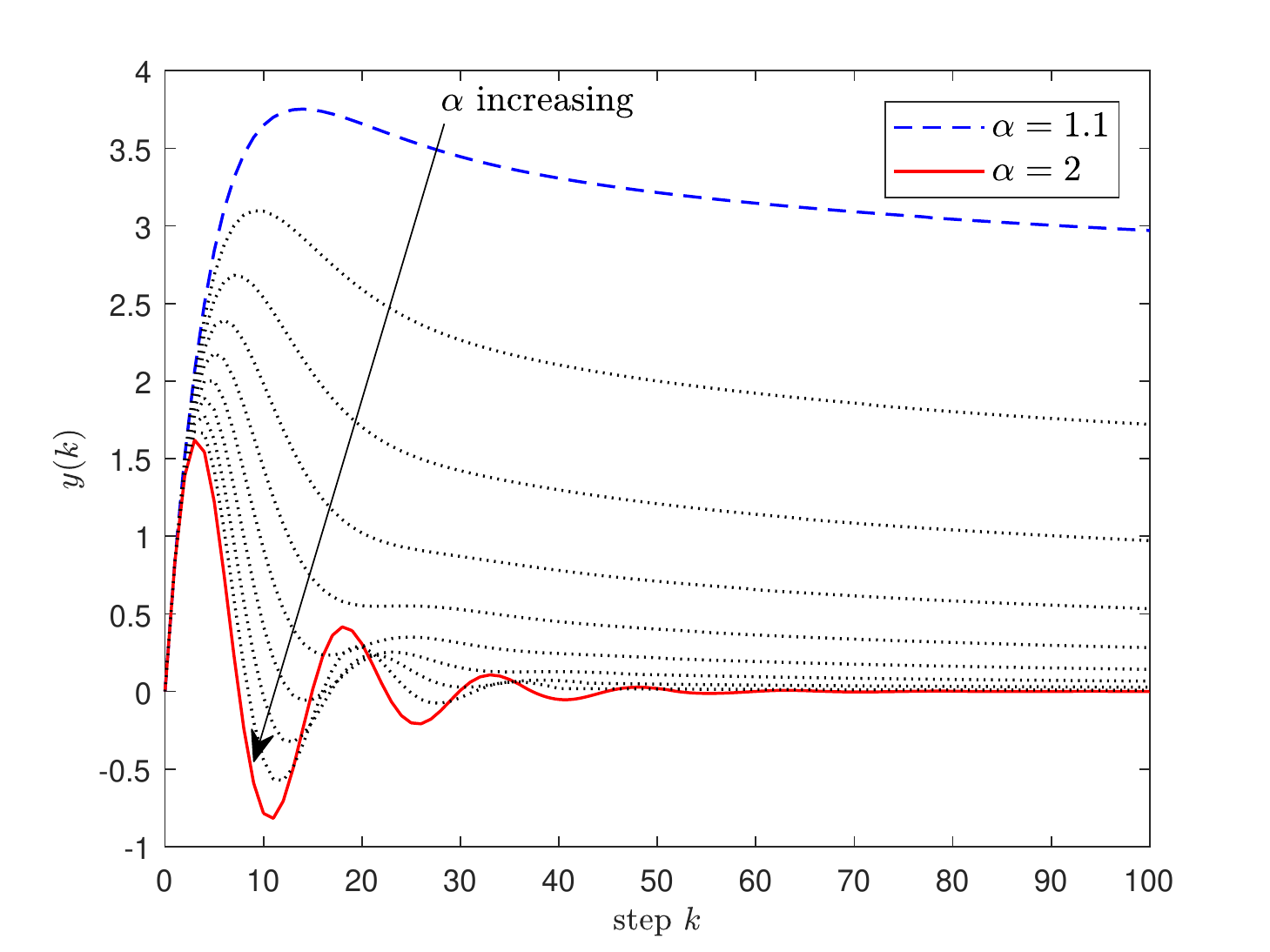}\vspace{-10pt}
\caption{Response $y(k)$ for different $\alpha\in(1,2]$ with $b_0=0$.}\label{Fig5}
\end{figure}
\begin{figure}[!htbp]
\centering
\includegraphics[width=0.65\columnwidth]{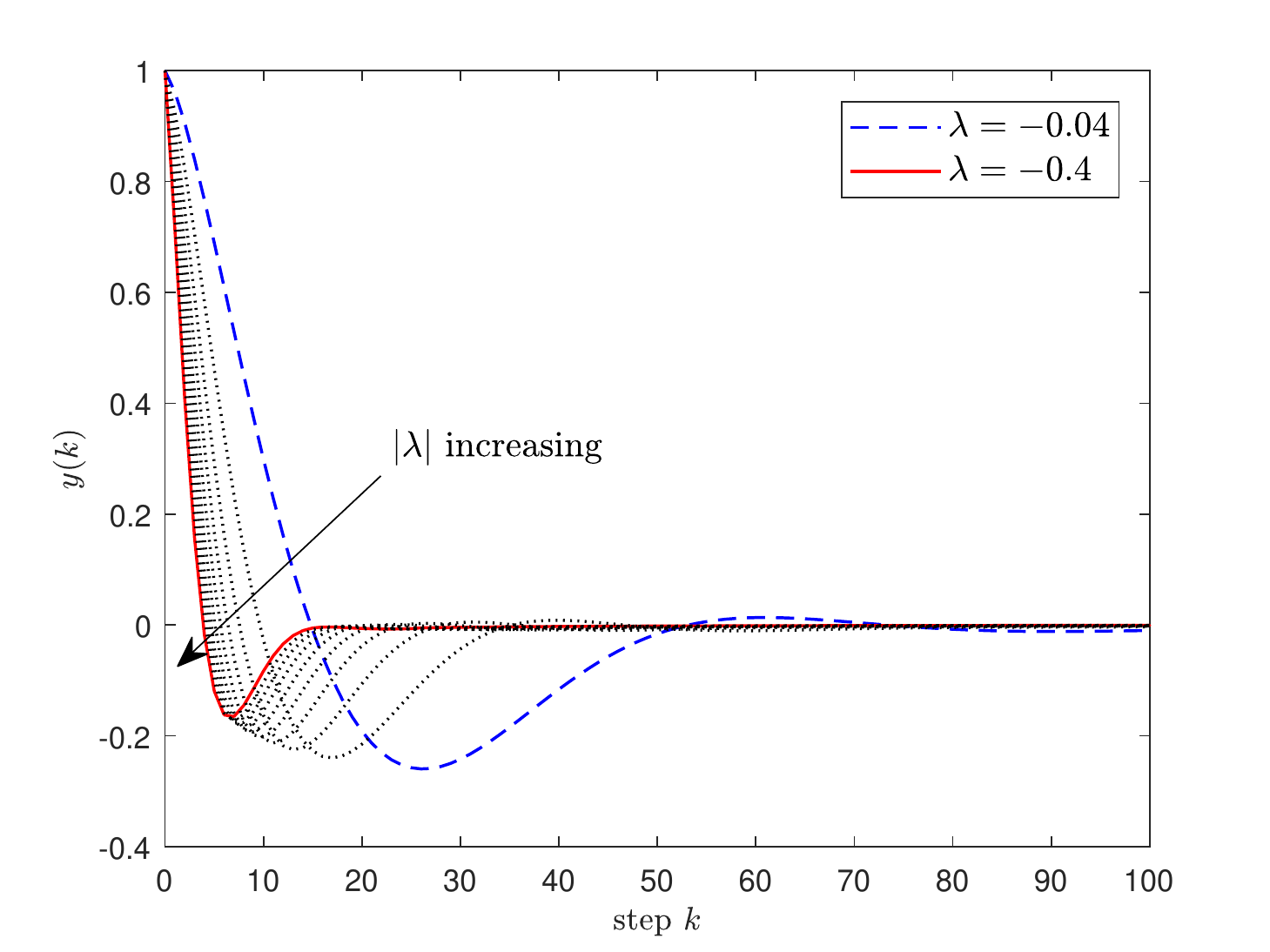}\vspace{-10pt}
\caption{Response $y(k)$ for different $\lambda<0$ with $b_1=0$.}\label{Fig6}
\end{figure}

One can observe from Fig. \ref{Fig3} that, the convergence speed becomes faster as $\alpha$ increases, when $\alpha\in (0,1]$. In particular, the system output $y(k)$ converges to $0$ at a very low speed, which coincides with that of the continuous-time case \cite{Gorenflo:2014Book}. It is shown in Fig. \ref{Fig4} that the rising time and overshoot both increase gradually as $\alpha$ increases within the interval $[1,2]$. Fig. \ref{Fig5} shows that the number of oscillation increases and the overshoot appears as $\alpha$ increases. In Fig. \ref{Fig6}, one can observe that the increase of $\left|\lambda\right|$ could speed up the convergence and meanwhile reduce the overshoot. These observations confirm our conclusions in Theorem \ref{Theorem 4}.

Besides, a number of simulations have been conducted setting $b_0\ne0$, $b_1\ne0$ and $1<\alpha\le2$. It is found that when $\alpha \in (1,1.7)$, the trajectory of $y(k)$ has no overshoot if $b_0$ and $b_1$ are chosen such that ${b_1}/{b_0} \in [1,9]$. Several questions then arise. Does there exist a constant $\alpha_c$ such that when $\alpha<\alpha_c$, the trajectory of $y(k)$ could have no overshoot by choosing $b_0$ and $b_1$ appropriately? How to choose $b_0$ and $b_1$, if such a constant $\alpha_c$ exists? To rigorously answer these questions will be the subjects of our on--going research.

%

\section{Conclusions}\label{Section 5}
The time--domain response of nabla discrete fractional order systems has been studied in this paper. Several useful properties of the nabla discrete Laplace transform and the discrete Mittag--Leffler function are explored, based on which an explicit solution to the system dynamic equation is presented and the time--domain response under zero input condition is comprehensively analyzed.

%
\section*{Acknowledgements}
The work described in this paper was supported by the National Natural Science Foundation of China (61601431, 61573332), the Anhui Provincial Natural Science Foundation (1708085QF141), the Fundamental Research Funds for the Central Universities (WK2100100028) and the General Financial Grant from the China Postdoctoral Science Foundation (2016M602032).

%

\phantomsection
\addcontentsline{toc}{section}{References}
\section*{References}
\bibliographystyle{model3-num-names}
\bibliography{datebase}

\end{document}